\documentclass{article}

\usepackage{amsmath, amssymb}
\usepackage[english]{babel}
\usepackage{paralist}
\usepackage[colorlinks=true]{hyperref}
\usepackage{array}
\usepackage[latin1]{inputenc}
\usepackage{xcolor}
\usepackage{esint}
\usepackage{latexsym,amsfonts}
\usepackage{amsthm}
\usepackage{array}
\usepackage{cleveref}
\usepackage{enumitem}
\usepackage{mathtools}
\usepackage{verbatim}
\usepackage{graphics}
\numberwithin{equation}{section}

\newtheorem{innercustomthm}{Theorem}
\newenvironment{thm-ref}[1]
  {\renewcommand\theinnercustomthm{#1}\innercustomthm}
  {\endinnercustomthm}
\newtheorem{innercustomcor}{Corollary}
\newenvironment{cor-ref}[1]
  {\renewcommand\theinnercustomcor{#1}\innercustomcor}
  {\endinnercustomcor}

\newtheorem{theorem}{Theorem}[section]
\newtheorem{lemma}[theorem]{Lemma}
\newtheorem{prop}[theorem]{Proposition}
\newtheorem{cor}[theorem]{Corollary}

\newtheorem*{Thm}{Theorem}
\newtheorem*{Lem}{Lemma}

\newtheorem*{main-thm}{Main Theorem}
\newtheorem*{main-lemma}{Main Lemma}

\theoremstyle{definition}
\newtheorem{defn}[theorem]{Definition}
\newtheorem{remark}[theorem]{Remark}

\newcommand{\abs}[1]{\left\vert#1\right\vert} 
\newcommand{\norm}[1]{\left\|#1\right\|} 
\newcommand{\pare}[1]{\left(#1\right)} 
\newcommand{\braces}[1]{\left\{#1\right\}} 
\newcommand{\brackets}[1]{\left[#1\right]} 
\newcommand{\set}[2]{\left\{#1 \; :\; #2\right\}} 

\DeclareMathOperator*{\diver}{div} 
\DeclareMathOperator*{\diam}{diam}
\DeclareMathOperator*{\dist}{dist}

\newcommand{\R}{\mathbb R} 
\newcommand{\N}{\mathbb N} 
\newcommand{\M}{\mathcal{M}} 
\renewcommand{\S}{\mathcal{S}}

\newcommand{\T}{\mathcal{T}}

\newcommand{\X}{\mathbb{X}}

\newcommand{\W}{\mathcal{W}}

\def\Xint#1{\mathchoice
   {\XXint\displaystyle\textstyle{#1}}%
   {\XXint\textstyle\scriptstyle{#1}}%
   {\XXint\scriptstyle\scriptscriptstyle{#1}}%
   {\XXint\scriptscriptstyle\scriptscriptstyle{#1}}%
   \!\int}
\def\XXint#1#2#3{{\setbox0=\hbox{$#1{#2#3}{\int}$}
     \vcenter{\hbox{$#2#3$}}\kern-.5\wd0}}

\def\dashint{\Xint-}

\begin{document}
\title{A priori H\"{o}lder and Lipschitz regularity for generalized $p$-harmonious functions in metric measure spaces}
\author{\'Angel Arroyo and Jos\'e G. Llorente}
\date{\small{ Departament de Matem\`{a}tiques \\ Universitat Aut\`onoma de Barcelona \\ 08193 Bellaterra. Barcelona \\SPAIN \\
arroyo@mat.uab.cat \\jgllorente@mat.uab.cat }\\ \ \\ \small{\today} }

\maketitle

\let\thefootnote\relax\footnotetext{\hspace{-7pt}\emph{Keywords:} mean value property, $p$-harmonious, $p$-laplacian, metric measure spaces.


Partially supported by grants MTM2011-24606, MTM2014-51824-p and 2014 SGR 75.}

\begin{abstract}
Let $(\X , d, \mu )$ be a proper metric measure space and let $\Omega \subset \X$ be a bounded domain. For each $x\in \Omega$, we choose a radius $0< \varrho (x) \leq \dist(x, \partial \Omega ) $ and let $B_x$ be the closed ball centered at $x$ with radius $\varrho (x)$. If $\alpha \in \R$, consider the following operator in $C( \overline{\Omega} )$, 
$$
				\T_{\alpha}u(x)=\frac{\alpha}{2}\pare{\sup_{B_x } u+\inf_{B_x } u}+(1-\alpha)\,\dashint_{B_x}\hspace{-0.1cm} u\ d\mu.
$$
Under appropriate assumptions on $\alpha$, $\X$, $\mu$ and the radius function $\varrho$ we show  that solutions   $u\in C( \overline{\Omega} )$ of the functional equation $\T_{\alpha}u = u$ satisfy  a local H\"{o}lder or Lipschitz condition in $\Omega$. The motivation comes from the so called $p$-harmonious functions in euclidean domains.
\end{abstract}


\section{Introduction}\label{SEC:INTRO}

The main goal of this paper is to provide a priori regularity estimates for functions satisfying certain nonlinear mean value properties  in metric measure spaces. Our main motivation are classical harmonic functions and the so called $p$-harmonious functions in $\R^n$. First of all, let us recall some basic facts about harmonic functions in euclidean space and their connections to the mean value property. 

It is well known that a continuous function $u$ in a domain $\Omega \subset \R^n$ is harmonic if and only if it satisfies the \emph{mean value property}
\begin{equation}\label {basic mvp}
u(x) = \dashint_{B(x,\varrho )} \hspace{-0.3cm} u \, dm
\end{equation}
for each $x\in \Omega$ and each $\varrho >0$ such that $ 0 < \varrho < \dist (x, \partial \Omega )$, where $m$ denotes $n$-dimensional Lebesgue measure. The mean value property plays a relevant role in Geometric Function Theory and is indeed the fundamental key of the interplay between harmonic functions,  Probability and Brownian motion. 
\\

The so called \emph{restricted} mean value property problems ask how many radii $\varrho$ in \eqref{basic mvp} are enough to guarantee harmonicity. One of the most representative results in this direction is  a classical theorem due to Volterra (for regular domains) and Kellogg's (in the general case): if $\Omega$ is bounded, $u\in C(\overline{\Omega})$ and if for each $x\in \Omega$ there is a radius $\varrho = \varrho (x)$, with $0< \varrho \leq \dist (x, \partial \Omega )$, such that \eqref{basic mvp} holds, then $u$ is harmonic in $\Omega$ (see \cite{V}, \cite{K}). Therefore, under appropriate hypothesis, the mean value property for a single radius (depending on the point) implies harmonicity. See \cite{N-V} for a detailed account of this and other results related to the mean value property.  
\\

The question of what are the natural stochastic processes associated to some nonlinear differential operators, like the $p$-laplacian or the $\infty$-laplacian, has attracted an increasing attention in the last years. If $1<p<\infty$ the \emph{$p$-laplacian} is the divergence form differential operator given by
$$
				\triangle_p u = \diver( \nabla u |\nabla u |^{p-2} )
$$
and weak solutions of the equation $\triangle_p u = 0$ are called  \emph{$p$-harmonic} functions. 
Suppose that $u\in C^2$  and that $\nabla u \neq 0$. Then direct computation gives 
\begin{equation} \label{p-lapla}
				\triangle_p u = |\nabla u |^{p-2} \pare{ \triangle u + (p-2)\frac{\triangle_{\infty}u}{|\nabla u |^2} }
\end{equation}
where 
$$
\triangle_{\infty}u = \sum_{i,j=1}^n u_{x_i}u_{x_j}u_{x_i , x_j}
$$
is the so called \emph{$\infty$-laplacian} in $\R^n$. So, at least in the smooth case and away from the critical points, the $p$-laplacian can be understood as a sort of linear combination of the usual laplacian and the normalized $\infty$-laplacian. Observe that we recover the usual laplacian when $p=2$. 
\\

Let us briefly explain now the connection between the $p$-laplacian and the mean value property. First, we recall that if $u\in C^2 (\Omega )$, where $\Omega \subset \R^n$, then the following asymptotic mean value property holds for any $x\in \Omega$:
\begin{equation} \label {Taylor}
				\lim_{\varrho \to 0} \frac{1}{\varrho^2} \pare{ \dashint_{B(x,\varrho )} \hspace{-0.3cm}u\, dm - u(x) } = \frac{\triangle u (x)}{2(n +2)}
\end{equation}
On the other hand, if $\nabla u(x) \neq 0$ then the following mid-range asymptotic mean value property also holds (see \cite{MAN-PAR-ROS-10}, \cite{LLO}): 

\begin{equation} \label {infty mvp}
				\lim_{\varrho \to 0} \frac{1}{\varrho^2} \brackets{\frac{1}{2}\pare{ \sup_{B(x,\varrho )}u + \inf_{B(x,\varrho)}u } - u(x) } = \frac{\triangle_{\infty}u(x)}{2|\nabla u (x) |^2}
\end{equation}
Therefore, taking
\begin{equation} \label{alpha}
\alpha = \frac{p-2}{p+n}
\end{equation}
it follows from \eqref{Taylor} and \eqref{infty mvp} that if $u\in C^2(\Omega )$ is $p$-harmonic in $\Omega$ then $u$ satisfies the \emph{asymptotic $p$-mean value property} 
\begin{equation} \label {p-mvp}
\lim_{\varrho \to 0}\frac{1}{\varrho^2}  \brackets{ \frac{\alpha}{2}\pare{ \sup_{B(x,\varrho )}u
+ \inf_{B(x,\varrho )}u }  + (1 - \alpha ) \dashint_{B(x,\varrho )} \hspace{-0.2cm}u
\, dm - u(x) } = 0
\end{equation}
at those $x$'s such that $\nabla u (x) \neq 0$. When $p \neq 2$ and $n\geq 3$ it is an open question whether  $p$-harmonic functions satisfy the asymptotic $p$-mean value property  at any point, one of the obstacles being that $p$-harmonic functions are only $C^{1, \beta}$  for some $0< \beta < 1$ (\cite{URA}, \cite{LEW}),  but not $C^2$ in general. More information has been recently obtained when $n=2$: it turns out that planar $p$-harmonic functions always satisfy the asymptotic $p$-mean value property at any point. (In \cite{LIN-MAN} the result was proven for a certain interval of $p$'s and in \cite{ARR-LLO-16-2} for the whole range $1 < p < \infty$).    
\\

Let $\Omega \subset \R^n$ be a bounded domain. Suppose that for each $x\in \Omega$, a radius $\varrho = \varrho (x)>0$ is given so that  $0 < \varrho \leq \dist(x, \partial \Omega )$ and let $B_x = \overline{B}(x,\varrho )$ be the closed ball centered at $x$ of radius $\varrho$. Inspection of formula (\ref{p-mvp}) suggests the definition of the following operators in $C(\overline{\Omega} )$:      
\begin{flalign*}
				\M u(x)&= \dashint_{B_x} u\ dm, \\
				\S u(x)&= \frac{1}{2}\, \pare{\sup_{B_x}u +   \inf_{B_x}u},  \\
				\T_\alpha u(x) &= \alpha\,\S u(x)+(1-\alpha)\M u(x).             
\end{flalign*}
The operators $\T_{\alpha}$ have recently come out in different contexts. When $\alpha = 1$ and a radius function in $\Omega$ is given, functions $u$ satisfying $\S u =u$ have been called \emph{harmonious functions} in the literature. The connection between harmonious functions and  extension problems was studied in \cite{LEG-ARC}, in the more general context of metric spaces. Existence and uniqueness of the Dirichlet problem for harmonious functions was also discussed there. The influential papers \cite{PER-SCH-SHE-WIL} and \cite{PER-SHE} opened the path to an stochastic interpretation of the $p$-laplacian and the $\infty$-laplacian, via the Dynamic Programming Principle (corresponding essentially to the functional equation $T_{\alpha} u = u $) for certain tug-of-war games. See also \cite{MAN-PAR-ROS-10} and \cite{MAN-PAR-ROS-12}, where the game-stochastic approach was continued and developed, in the case $0 \leq \alpha < 1$, or $p\geq 2$.  
\\

If $p \geq 2$, $\alpha$ is as in (\ref{alpha}) and  $r(x) = \varepsilon$ is constant then (not necessarily continuous) functions $u$ satisfying $\T_{\alpha}u = u $ were called \emph{$p$-harmonious functions} in \cite{MAN-PAR-ROS-12}. Note that the range $0\leq \alpha \leq 1$ corresponds to the range $2 \leq p \leq +\infty$.  In order to pose the Dirichlet problem for such $p$-harmonious functions,  the authors in \cite{MAN-PAR-ROS-12} needed to extend a given $f\in C(\partial \Omega ) $  to the strip $\displaystyle \{ x\in \R^n \setminus \Omega \, : \dist(x, \partial \Omega ) \leq \varepsilon \} $ and proved that, if $\Omega \subset \R^n$ is bounded and satisfies some regularity assumptions  then there is a unique $p$-harmonious function $u_{\varepsilon}$ having $f$ as boundary values (in the extended sense). Furthermore, $\{ u_{\varepsilon} \} \to u$ uniformly in $\overline{\Omega}$ as $\varepsilon \to 0$, where $u$ is the unique $p$-harmonic function solving the Dirichlet problem in $\Omega$ with boundary data $f$. See also \cite{LUI-PAR-SAK-1} for an analytic approach, still in the constant radius case. 
\\

Continuous functions $u$ satisfying  $\T_{\alpha}u = u$ in the variable radius case were considered in \cite{ARR-LLO-16-1} and the existence and uniqueness of the Dirichlet problem for such a class of functions was established there under certain assumptions on the domain, the parameter $\alpha$ and the radius function. 
\\

Our main concern in this paper is to provide  H\"{o}lder and Lipschitz regularity estimates for continuous solutions of the functional equation $\T_{\alpha}u = u$ in metric measure spaces, depending on the regularity of the radius function $\varrho$ (see \Cref{main} below). In the constant radius case, the local Lipschitz regularity of $p$-harmonious functions for $p \geq 2$ was obtained in \cite{LUI-PAR-SAK-2}. As for the case $\alpha =1$ (or $p = \infty$),  not much is known. Unfortunately, our methods cannot be extended to cover the case $\alpha =1$. 


\section{Preliminary definitions and main results}

\subsection{Metric measure spaces and admissible radius functions}

Let $(\X , d)$ be a metric space. We say that $(\X , d)$ is \emph{proper} if every closed and bounded subset of $\X$ is compact. $(\X , d)$ is a \emph{geodesic space} if for any two points $x$, $y\in \X$ there is a curve $\gamma$ connecting $x$ and $y$ whose length is equal to $d(x,y)$. 
\\

A \emph{metric measure space} $(\X , d, \mu )$ is a metric space endowed with a Borel positive regular measure $\mu$. In what follows, we will only consider measures $\mu$ such that  $0 < \mu (B) < \infty$ for every ball $B \subset \X$. 
\\

\begin{defn}
Let $(\X,d,\mu)$ be a metric measure space. We say that $\mu$ is \emph{doubling} (equivalently, $(\X,d,\mu)$ is a \emph{doubling metric measure space}) if there exists a constant $D_\mu\geq1$ such that
\begin{equation}\label{doubling}
				\mu\pare{B(x,2r)} \leq D_\mu\, \mu(B(x,r))
\end{equation}
for any $x\in\X$ and each $r>0$.
\end{defn}

The following property will play a central role in what follows.  

\begin{defn}
Let $\delta\in(0,1]$. A metric measure space $(\X,d,\mu)$ satisfies the \emph{$\delta$-annular decay property} if there exists a constant $D_\delta\geq 1$ such that
\begin{equation}\label{delta-AD}
				\mu\pare{B(x,R)\setminus B(x,r)} \leq D_\delta\pare{\frac{R-r}{R}}^\delta \mu(B(x,R)),
\end{equation}
for each $x\in\X$ and $0<r\leq R$. For $\delta=1$, this property is also known as the \emph{strong annular decay property}.
\end{defn}

We will also use the following definition when studying the continuity properties of the operator $\M$. 

\begin{defn}\label{DEF-ring-cont}
We say that a (Borel, regular) measure $\mu$ in a metric space $\X$ is \emph{ring-continuous} if, for each $x\in \X$ the function 
\begin{equation*}
r\longmapsto\mu(B(x,r))
\end{equation*}
is continuous in $(0, +\infty )$. 
\end{defn}

As a canonical example, $\R^n$ endowed with the euclidean distance and Lebesgue $n$-dimensional measure satisfies the strong annular decay property. The $\delta$-annular decay property was introduced in manifolds by Colding and Minicozzi (\cite{C-M}) and, independently, in metric spaces by Buckley (\cite{BUC}). It is easy to check that the $\delta$-annular decay property implies the doubling property. Conversely, in \cite{BUC} it is proved in particular that a geodesic metric space $( \X, d, \mu )$ with a doubling measure $\mu$ satisfies a $\delta$-annular decay condition for some $\delta \in (0,1]$, where $\delta$ only depends on the doubling constant. In the context of harmonicity in metric measure spaces, the $\delta$-annular decay property has already been used in \cite{ADA-GAC-GOR}. See also \cite{B-B-L} for a local version.

\begin{remark}
Let $(\X,d,\mu)$ be a metric measure space. The following implications hold:
\begin{equation*}
\begin{array}{ccc}
				\delta\mbox{-Annular Decay}  & \Longrightarrow & \mbox{Ring-continuous} \\
				\Downarrow & & \\
				\mbox{Doubling Property} & &
\end{array}
\end{equation*}
In addition, by \cite{DAN-GAR-NHI} and \cite{DIF-GUT-LAN},  if $(\X,d,\mu)$ is geodesic then
\begin{equation*}
				\mbox{Doubling Property}\quad\Longrightarrow\quad\mbox{Ring-continuous}.
\end{equation*}
Moreover, by \cite{BUC}, if $(\X,d,\mu)$ is geodesic then
\begin{equation*}
				\mbox{Doubling Property}\quad\Longrightarrow\quad\delta\mbox{-Annular Decay}.
\end{equation*}
\end{remark}

\

We introduce some basic concepts that will be useful in the following sections: given any subset $G\subset\X$, we denote by $\dist(x,G)$ the infimum of all distances $d(x,y)$ where $y\in G$. Moreover, if $G$ is bounded, let $\ell(G)$ be the largest distance to the boundary for points in $G$:
\begin{equation}\label{eq-L-Omega}
		\ell(G):\,=\sup_{x\in G}\braces{\dist(x,\partial G)}\leq\frac{1}{2}\diam G.
\end{equation}

Given two subsets $A,B\subset\X$, we denote by  $A\triangle B=(A\setminus B)\cup(B\setminus A)$  the symmetric difference of $A$ and $B$. If $A,B,C\subset\X$, it follows that
\begin{equation*}
				A\triangle B=(A\triangle C)\triangle(C\triangle B)\subset(A\triangle C)\cup(C\triangle B).
\end{equation*}
If $A,B\subset\X$ are two measurable subsets, then
\begin{equation*}
				\abs{\mu(A)-\mu(B)}\leq\mu(A\triangle B).
\end{equation*}
and, from the triangle inequality,
\begin{equation}\label{mu-triangle}
				\mu(A\triangle B)\leq\mu(A\triangle C)+\mu(C\triangle B).
\end{equation}

A \emph{modulus of continuity} in a bounded domain $\Omega \subset \X$ is a non-decreasing continuous function $\omega:[0,\diam\Omega] \rightarrow [0,\infty)$ such that $\omega(0)=0$. We will often require  $\omega$ to be concave too. If $G \subset \Omega $  and  $u\in C(\overline{G})$, we will denote by $\omega_{u,G}$ a concave modulus of continuity such that
\begin{equation}\label{def-mod-cont}
				\abs{u(x) - u(y)}\leq\omega_{u,G}(d(x,y))
\end{equation}
for all $x,y\in G$.

\begin{defn}
Let $\Omega\subset\X$ be a fixed bounded open domain in a proper metric space $\X$. We say that a non-negative function $\varrho\in C(\overline\Omega)$ is an \emph{admissible radius function in $\Omega$} if $\displaystyle 0<\varrho(x)\leq \dist(x,\partial\Omega)$ for each $x\in\Omega$, and $\varrho(x)=0$ if and only if $x\in\partial\Omega$. Whenever $G\Subset\Omega$, we define 
\begin{equation} \label{rho G}
\varrho_G:\,=\inf_G\varrho>0.
\end{equation} 
Also, we introduce the following notation for closed balls in $\Omega$ with radii given by $\varrho$:
\begin{equation*}
				B_x :\,= \overline{B}(x,\varrho(x))
\end{equation*}
for each $x\in\Omega$. Since the balls $B_x$ are not necessarily contained in $G$, we define
\begin{equation}\label{hull}
				\widetilde{G}:\,=\bigcup_{x\in G}B_x.
\end{equation}
\end{defn}

Following the notation in \eqref{def-mod-cont}, we denote by $\omega_{\varrho,\Omega}$  a concave modulus of continuity for $\varrho$ in $\Omega$. Since $|\varrho (x) - \varrho (y) | \leq  \diam\Omega$ for each $x$, $y\in \Omega$,  we can also assume that  $\omega_{\varrho,\Omega} (\diam\Omega) \leq  \diam\Omega$. As we will see in the next sections, a distinguished case occurs when the admissible radius function is $L$-Lipschitz, that is,
\begin{equation*}
				\abs{\varrho(x)-\varrho(y)}\leq  L \, d(x,y),
\end{equation*}
for each $x,y\in\Omega$ , in which case we can simply take $\omega_{\varrho,\Omega}(t)= Lt$. For technical reasons, we need to define another concave modulus of continuity for $\varrho$ (that will be  denoted by $\widehat\omega_\varrho$) as follows: if $\omega_{\varrho,\Omega}(t) \leq t$ for all $t \in [0, \diam\Omega]$ then we set $\widehat\omega_\varrho (t) : \, = t $. Otherwise, we define  
\begin{equation}\label{def-hat-rho}
				\widehat\omega_\varrho(t):\,=\frac{\diam\Omega}{\omega_{\varrho,\Omega}(\diam\Omega)}\,\omega_{\varrho,\Omega}(t).
\end{equation}
Note that, defined in this way, $\widehat\omega_\varrho(t)$ is a concave modulus of continuity for $\varrho$ in $\Omega$ satisfying
\begin{equation}\label{hat-rho-bounds}
				\max\braces{t,\omega_{\varrho,\Omega}(t)}\leq\widehat\omega_\varrho(t)\leq\diam\Omega = \widehat\omega_\varrho (\diam\Omega)
\end{equation}
for each $t\in[0,\diam\Omega]$. Consequently, successive compositions of $\omega_\varrho$ with itself will produce a sequence of continuous functions $\widehat\omega_\varrho^{(n)}:[0,\diam\Omega]\rightarrow[0,\diam\Omega]$ given by
\begin{equation*}
				\widehat\omega_\varrho^{(n)}(t):\,=\widehat\omega_\varrho\pare{\widehat\omega_\varrho^{(n-1)}(t)},
\end{equation*}
for $n\in\N$, where $\widehat\omega_\varrho^{(0)}(t)=t$.


\begin{remark}
We will hereafter make use of some of the concepts introduced in this subsection (like the family of balls $\set{B_x}{x\in\Omega}$ and the operator on sets $\widetilde{(\cdot)}$) without any explicit mention of their dependence on the choice of the admissible radius function $\varrho$, which is assumed to be fixed. 
\end{remark}


\subsection{Main results}

Let $(\X,d,\mu)$ be a metric measure space. Assume that an admissible radius function $\varrho$ in a domain $\Omega\subset\X$ is given. If $u\in C(\overline{\Omega})$, $x\in\Omega$ and $\alpha\in\R$ we define:
\begin{flalign}
				\M u(x)&:\,= \dashint_{B_x} u\ d\mu, \label{def-M} \\
				\S u(x)&:\,= \frac{1}{2}\, \pare{\sup_{B_x}u + \,  \inf_{B_x}u}, \label{def-S}  \\
				\T_\alpha u(x) & :\,= \alpha\,\S u(x)+(1-\alpha)\M u(x). \label{def-T}             
\end{flalign}
%
%
%
We are interested in studying the fixed points of the operators $\T_{\alpha}$, which can be seen as functions satisfying an specific nonlinear mean value property. For that reason, we give the following fundamental definition. 

\begin{defn} \label{def alpha mvp}
Let $(\X,d,\mu)$ be a metric measure space, $\Omega\subset\X$ a domain and  $\varrho$ an admissible radius function in $\Omega$. Let $\alpha\in \R$.  A function $u\in C(\overline{\Omega})$ is said to satisfy the \emph{$\alpha$-mean value property} in $\Omega$ if it is a solution of the functional equation
\begin{equation*}
				\T_\alpha u= u.
\end{equation*}

\end{defn}

\

The case $\alpha = 0$ is interesting enough by itself. Harmonicity in a metric measure space $\X$ in connection to the mean value property has been recently introduced in \cite{GAC-GOR} and \cite{ADA-GAC-GOR} in the following way: a locally integrable function in a domain $\Omega \Subset \X $ is said \emph{strongly harmonic} in $\Omega $ if it satisfies the mean value property in any ball compactly contained in $\Omega$. The following regularity result has been obtained in  \cite{ADA-GAC-GOR}:  

\begin{Thm}[{\cite[Thm. 4.2]{ADA-GAC-GOR}}] 
If $(\X , d, \mu )$ is a doubling metric measure space satisfying a $\delta$-annular decay condition for some $\delta \in (0,1] $  then every locally bounded, strongly harmonic function  $u$ in a domain $\Omega \subset \X$ is locally $\delta$-H\"{o}lder continuous in $\Omega$. In particular, if $\delta = 1$ then  $u$ is locally Lipschitz continuous in $\Omega$. 
\end{Thm}

(See also Lemma 2.3 in \cite{ARR-LLO-16-1}, where the local H\"older continuity of functions satisfying the mean value property  for a single radius is obtained if $\X = \R^n$, $\mu$ is doubling and the radius function is $1$-Lipschitz). \\

We have obtained the following generalizations for functions satisfying the $0$-mean value property in the sense of \Cref{def alpha mvp} with respect to some admissible radius function.  

\begin{cor-ref}{\ref{M regularity}}
Let $(\X,d,\mu)$ be a proper metric measure space satisfying the $\delta$-annular decay property for some $\delta\in(0,1]$. Suppose that there is $\gamma \in (0,1]$ such that $\varrho$ is a $\gamma$-H\"{o}lder continuous admissible radius function in a bounded domain $\Omega\subset\X$. Then any $u\in L^{\infty} (\Omega )$ satisfying the $0$-mean value property in $\Omega$ with respect to the radius admissible function $\varrho$ (that is, $\M  u = u$) is locally $ \gamma \delta $-H\"{o}lder continuous in $\Omega$. In particular, if  $\delta = \gamma = 1$ then $u$ is locally Lipschitz continuous in $\Omega$.  
\end{cor-ref}

As for the general case $\alpha \neq 0$, our main result requires certain rigid control of the radius function. 

\begin{thm-ref}{\ref{main}}
Let $(\X,d,\mu)$ be a proper, geodesic metric measure space satisfying the $\delta$-annular decay condition for some $\delta\in(0,1]$ and let $\Omega \subset \X$ be a bounded domain. Suppose that  $\varrho$ is a Lipschitz admissible radius function in $\Omega$ with Lipschitz constant $L \geq 1$ such that
\begin{equation*}
				\lambda\dist(x,\partial\Omega)^\beta\leq\varrho(x)\leq\varepsilon\dist(x,\partial\Omega),
\end{equation*}
for all $x\in \Omega$, where $0 < \lambda \leq \ell(\Omega )^{1-\beta}\varepsilon$. Assume also that  
\begin{equation*}
\begin{split}
				\abs{\alpha} & < L^{-1},  \\
				0 <  \ \varepsilon ~& < 1-L\abs{\alpha},
\end{split}
\end{equation*}
and choose $\beta$ so that 
\begin{equation*} 
1  \, \leq \, \beta<\frac{\log\displaystyle\frac{1}{L\abs{\alpha}}}{\, \log\displaystyle\frac{1}{1-\varepsilon}\, }. 
\end{equation*}
Then any $u\in C(\overline{\Omega})$ verifying the $\alpha$-mean value property in $\Omega$ with respect to $\varrho$ (that is, $\T_{\alpha} u = u$) is locally $\delta$-H\"{o}lder continuous in $\Omega$. In particular, if $\delta = 1$ then $u$ is locally Lipschitz continuous in $\Omega$. 
\end{thm-ref}

In the particular case $\beta = 1$ we get the following corollary.

\begin{cor-ref}{\ref{reg beta=1}}
Let $(\X,d,\mu)$ be a proper, geodesic metric measure space satisfying the
$\delta$-annular decay condition for some $\delta\in(0,1]$ and let $\Omega \subset \X$ be a bounded domain. Suppose that  $\varrho$ is a Lipschitz admissible radius function in $\Omega$ with Lipschitz constant $L \geq 1$ such that
\begin{equation*}
				\lambda\dist(x,\partial\Omega)\leq\varrho(x)\leq\varepsilon\dist(x,\partial\Omega),
\end{equation*}
for all $x\in \Omega$, where $0 < \lambda \leq \varepsilon$. Assume also that  
\begin{equation*}
\begin{split}
				\abs{\alpha} & < L^{-1},  \\
				0 <  \ \varepsilon ~& < 1-L\abs{\alpha}.
\end{split}
\end{equation*}
Then any $u\in C(\overline{\Omega})$ verifying the $\alpha$-mean value property in $\Omega$ with respect to $\varrho$ (that is, $\T_{\alpha} u = u$) is locally $\delta$-H\"{o}lder continuous in $\Omega$. In particular, if $\delta = 1$ then $u$ is locally Lipschitz continuous in $\Omega$. 
\end{cor-ref}

We  obtain further regularity for solutions of the $\alpha$-mean value property assuming that they are continuous in $\overline{\Omega}$. This explains the \textit{a priori} in the title. However, the existence part is not discussed here. Compare with \cite{ARR-LLO-16-1}, where existence and uniqueness of the Dirichlet problem are established if  $\X = \R^n$, $\Omega \subset  \R^n$ is bounded and strictly convex and $\mu$ is Lebesgue measure. In this particular case, the connection between $p$-harmonious functions and the  $\alpha$-mean value property has already been pointed out at the introduction,  where $\alpha$ and $p$ are related by \eqref{alpha} (note that the intervals $1 < p < \infty$ and $2 \leq p < \infty$ correspond, respectively, to the intervals $-\frac{1}{n+1} < \alpha < 1$ and $0 \leq \alpha < 1$). This explains the term \textit{generalized $p$-harmonious} in the title, even though the link between $p$ and $\alpha$ is missing in the general metric space case.


\section{Basic Estimates for $\M$ and $\S$}

\subsection{Continuity of $\M$}

We will first look at the continuity and regularity of the function
\begin{equation}\label{def-pre-M}
				x\longmapsto\M u(x)=\dashint_{B_x}u\ d\mu
\end{equation}
where an admissible radius function $\varrho$ in a domain $\Omega\subset\X$, a measure $\mu$ and a bounded, continuous function $u$ in $\Omega$ are given. The following Lemma is a preliminary result in this direction.

\begin{lemma}\label{LEMMA-Mu}
Let $(\X,d,\mu)$ be a metric measure space. If $B_1$ and $B_2$ are two balls contained in $\X$, then
\begin{equation}\label{M-est-12}
				\abs{\dashint_{B_1} u\ d\mu-\dashint_{B_2} u\ d\mu} \leq 2\norm{u}_\infty \frac{\mu(B_1\triangle B_2)}{\max\braces{\mu(B_1),\mu(B_2)}},
\end{equation}
for each $u\in L^\infty(\X)$.
\end{lemma}

\begin{proof}
We can assume that $\mu(B_1)\geq\mu(B_2)$, then
\begin{multline*}
				\mu(B_1)\pare{\dashint_{B_1} u\ d\mu - \dashint_{B_2}u\ d\mu} \\
				= \int_{B_1} u\ d\mu - \int_{B_2}u\ d\mu+\pare{\mu(B_2)-\mu(B_1)}\dashint_{B_2}u\ d\mu,
\end{multline*}
and estimating this, we obtain
\begin{equation*}
\begin{split}
				\mu(B_1)\abs{\dashint_{B_1} u\ d\mu - \dashint_{B_2}u\ d\mu}
                ~& \leq \abs{\int_{B_1} u\ d\mu - \int_{B_2}u\ d\mu}+\norm{u}_\infty \abs{\mu(B_2)-\mu(B_1)} \\
 				~& \leq \int_{B_1\triangle B_2} \abs{u}\ d\mu+\norm{u}_\infty \mu(B_1\triangle B_2) \\
				~& \leq 2\norm{u}_\infty\mu(B_1\triangle B_2).
\end{split}
\end{equation*}
\end{proof}

The following corollary follows from \Cref{LEMMA-Mu} and the fact that $|| \M v ||_{\infty} \leq ||v||_{\infty}$ if $v\in L^{\infty}(\Omega )$. 
\begin{cor}\label{RES-M-EST}
Let $(\X,d,\mu)$ be a metric measure space. Let $\Omega\subset\X$ be a domain and $\varrho$ an admissible radius function in $\Omega$. Then, for each $u\in L^\infty(\Omega)$ and all $x$, $y\in \Omega$ we have
\begin{equation}\label{M-est}
				\abs{\M^n u(x)-\M^n u(y)} \leq 2\norm{u}_\infty \frac{\mu(B_x\triangle B_y)}{\max\braces{\mu(B_x),\mu(B_y)}}.
\end{equation}
\end{cor}


The importance of \Cref{RES-M-EST} lies in the fact that the continuity of $\M u$ can be transferred from the continuity of the function
\begin{equation}\label{mu-Bx-function}
				x\longmapsto\mu(B_x)=\mu(B(x,\varrho(x)),
\end{equation}
without any dependence of the  function $u$. To see that,  consider any $x,y\in\Omega$ and $r_1,r_2>0$ and recall \eqref{mu-triangle}. Then
\begin{equation}\label{ineq-aux-001}
				\mu(B(x,r_1)\triangle B(y,r_2)) \leq\mu(B(x,r_1)\triangle B(x,r_2)) + \mu(B(x,r_2)\triangle B(y,r_2)).
\end{equation}
%
Now suppose that $\mu$ is ring-continuous (recall \Cref{DEF-ring-cont}). Then, since $B(x,r_1)\subset B(x,r_2)$ or $B(x,r_2)\subset B(x,r_1)$,  the first term in the right hand side of \eqref{ineq-aux-001} is equal to $\abs{\mu(B(x,r_1))-\mu(B(x,r_2))}$. For the second term, we recall the following result due to Gaczkowski and G\'orka:

\begin{Lem}[{\cite[Theorem 2.1]{GAC-GOR}}]
Let $(\X,d,\mu)$ be a metric measure space such that $\mu$ is ring-continuous. Then for each  $x\in\X$ and each $r>0$,
\begin{equation}\label{mu-cont-wrt-d}
				\lim_{y\rightarrow x}\mu\pare{B(x,r)\triangle B(y,r)}=0.
\end{equation}
Moreover, the function $x\mapsto\mu(B(x,r))$ is continuous (w.r.t. $d$) for each fixed $r>0$.
\end{Lem}

\begin{remark}
The converse is not true (see Example 2 in \cite{ADA-GAC-GOR}). 
\end{remark}

Therefore, replacing $r_1=\varrho(x)$ and $r_2=\varrho(y)$ in \eqref{ineq-aux-001} we get the following proposition. 

\begin{prop}\label{RES-M-CONT}
Let $(\X,d,\mu)$ be a metric measure space such that $\mu$ is ring-continuous. Suppose that $\Omega\subset\X$ is a domain and $\varrho$ is a continuous admissible radius function in $\Omega$. Then, $\M:L^\infty(\Omega)\rightarrow C(\Omega)$.
\end{prop}

\begin{remark}
By definition, the continuous admissible radius function $\varrho$ vanishes on the boundary of the domain $\Omega$, thus $\mu(B_x)$ tends to zero as $x$ approaches the boundary of $\Omega$. In consequence, estimates obtained from \eqref{M-est} are local, that is, they only make sense on compact subsets $K\subset\Omega$.
\end{remark}

\subsection{Estimates for $\M$}

Let $\Omega\subset\X$ be a given domain in a metric measure space $(\X,d,\mu)$ and let $K\subset\Omega$ be a compact subset. In this section we will construct moduli of continuity $\W_{\mu,K}$ depending on $\mu$, $\varrho$ and $K$ such that
\begin{equation}\label{def-mod-WK-0}
				\frac{\mu(B_x\triangle B_y)}{\max\braces{\mu(B_x),\mu(B_y)}}\leq\frac{1}{2}\,\W_{\mu,K}(d(x,y)),
\end{equation}
for every $x,y\in K$. Hence, by \eqref{M-est}, we would have
\begin{equation}\label{def-mod-WK-1}
				\abs{\M^nu(x)-\M^nu(y)}\leq \norm{u}_\infty\W_{\mu,K}(d(x,y)),
\end{equation}
for each $n\in\N$.

\begin{lemma}\label{RES-RC-LIP}
Let $(\X,d,\mu)$ be a metric measure space satisfying the $\delta$-annular decay property \eqref{delta-AD} for some $\delta\in(0,1]$ and $D_\delta\geq 1$. Suppose that $\varrho$ is a $L$-Lipschitz admissible radius function in a domain $\Omega\subset\X$ for some $L\geq 1$. Then, for any compact set $K\subset\Omega$ and each $x$, $y\in K$ we have
\begin{equation}\label{eq-mod-cont-delta}
				\frac{\mu(B_x\triangle B_y)}{\max\braces{\mu(B_x),\mu(B_y)}}\leq
                4\,L\,\,D_\delta\pare{\frac{d(x,y)}{\varrho_K}}^\delta.
\end{equation}
\end{lemma}

\begin{proof}
Since $\varrho$ is $L$-Lipschitz by assumption, $\abs{\varrho(x)-\varrho(y)}\leq L\, d(x,y)$. Then:
\setlength{\leftmargini}{14pt}
\begin{enumerate}[label=\roman*)]
\item if $d(x,y)>\displaystyle\frac{\varrho_K}{2L}$, then $\displaystyle D_\delta\pare{\frac{2L\,d(x,y)}{\varrho_K}}^\delta>1$, and
\begin{equation*}
\begin{split}
				\mu(B_x\triangle B_y) ~& \leq 2\max\braces{\mu(B_x),\mu(B_y)} \\
				~& <2\,D_\delta\pare{\frac{2L\,d(x,y)}{\varrho_K}}^\delta\max\braces{\mu(B_x),\mu(B_y)}.
\end{split}
\end{equation*}
\item If $d(x,y)\leq\displaystyle\frac{\varrho_K}{2L}$ then, since $L\geq 1$, we get that $\abs{\varrho(x)-\varrho(y)}\leq\varrho_K/2$ and, in particular, $\varrho(y)\geq\varrho(x)/2$ and $\varrho(x)\geq\varrho(y)/2$. As a consequence, the following inclusions hold:
\begin{equation*}
\begin{split}
				B_x \setminus B_y & \subset B_x \setminus B(x,\varrho(y) - d(x,y)), \\
				B_y \setminus B_x & \subset B_y \setminus B(y,\varrho(x) - d(x,y)).
\end{split}
\end{equation*}
Thus, by \eqref{delta-AD} and the fact that $\varrho(x),\varrho(y)\geq\varrho_K$ for $x,y\in K$, we obtain
\begin{equation*}
\begin{split}
				\mu(B_x \setminus B_y)\leq D_\delta\pare{\frac{\varrho(x)-\varrho(y)+d(x,y)}{\varrho_K}}^\delta\max\braces{\mu(B_x),\mu(B_y)}, \\
				\mu(B_y \setminus B_x)\leq D_\delta\pare{\frac{\varrho(y)-\varrho(x)+d(x,y)}{\varrho_K}}^\delta\max\braces{\mu(B_x),\mu(B_y)}.
\end{split}
\end{equation*}
Using the $L$-Lipschitz assumption on $\varrho$ and adding these two quantities we get
\begin{equation*}
				\mu(B_x \triangle B_y)\leq 2\,D_\delta\pare{\frac{(L+1)\,d(x,y)}{\varrho_K}}^\delta\max\braces{\mu(B_x),\mu(B_y)},
\end{equation*}
which implies \eqref{eq-mod-cont-delta}.
\end{enumerate}
\end{proof}

\begin{remark}
Note that if $x,y$ are as in the statement of \Cref{RES-RC-LIP} then only the pointwise inequality $\abs{\varrho(x)-\varrho(y)}\leq L\, d(x,y)$ is really used in the proof.
\end{remark}

\begin{lemma}
Let $(\X,d,\mu)$ be a proper metric measure space satisfying the $\delta$-annular decay property \eqref{delta-AD} for some $\delta\in(0,1]$ and $D_\delta\geq 1$. Suppose that $\varrho$ is a continuous admissible radius function in a bounded domain $\Omega\subset\X$. Then, for any compact set $K\subset\Omega$ and each $x$,$y\in K$ we have
\begin{equation}\label{eq-mod-cont-doubling}
				\frac{\mu(B_x\triangle B_y)}{\max\braces{\mu(B_x),\mu(B_y)}}\leq
                C\pare{\frac{\widehat{\omega}_\varrho(d(x,y))}{\varrho_K}}^\delta,
\end{equation}
where $C=C(D_\delta, \mu) >0$ and $\widehat{\omega}_\varrho$ is as in  \eqref{hat-rho-bounds}. 
\end{lemma}

\begin{proof}
Since $\varrho$ is a continuous function by assumption, for each pair of points $x,y\in K$, we need distinguish two cases depending on the values of $\abs{\varrho(x)-\varrho(y)}$: if $\abs{\varrho(x)-\varrho(y)}\leq d(x,y)$, this case was already studied in \Cref{RES-RC-LIP} with $L=1$, then \eqref{eq-mod-cont-doubling} follows from \eqref{eq-mod-cont-delta} and \eqref{hat-rho-bounds}.

Otherwise, $\abs{\varrho(x)-\varrho(y)}>d(x,y)$. We can assume directly that
\begin{equation}\label{eq-assumption}
				d(x,y)<\varrho(x)-\varrho(y),
\end{equation}
since the other case is analogous. Then $B_y\subset B_x$ and
\begin{equation*}
				B_x\triangle B_y=B_x\setminus B_y\subset B(y,\varrho(x)+d(x,y))\setminus B(y,\varrho(y)).
\end{equation*}
Consequently, the $\delta$-annular decay \eqref{delta-AD} yields
\begin{equation}\label{eq-aux-01}
				\mu(B_x\triangle B_y)\leq D_\delta\pare{\frac{\varrho(x)-\varrho(y)+d(x,y)}{\varrho(x)+d(x,y)}}^\delta\mu(B(y,\varrho(x)+d(x,y))).
\end{equation}
On the other hand, since the $\delta$-annular decay property implies that $\mu$ is doubling with some constant $D_\mu\geq 1$, using the inclusion $B(y,\varrho(x)+d(x,y))\subset B(y,2\varrho(x))$, it turns out that
\begin{equation*}
				\mu(B(y,\varrho(x)+d(x,y)))\leq D_\mu^2\,\mu(B_x).
\end{equation*}
Therefore, replacing this in \eqref{eq-aux-01} we reach
\begin{equation*}
				\mu(B_x\triangle B_y)\leq D_\mu^2\,D_\delta\pare{\frac{\varrho(x)-\varrho(y)+d(x,y)}{\varrho(x)+d(x,y)}}^\delta\mu(B_x).
\end{equation*}
Since $d(x,y)\geq 0$, $\varrho(x)\geq\varrho_K$, $\mu(B_x)\geq\mu(B_y)$ and \eqref{eq-assumption},
\begin{equation*}
				\mu(B_x\triangle B_y)\leq D_\mu^2\,D_\delta\pare{2\,\frac{\varrho(x)-\varrho(y)}{\varrho_K}}^\delta\max\braces{\mu(B_x),\mu(B_y)}.
\end{equation*}
Recalling \eqref{hat-rho-bounds} the proof is completed.
\end{proof}

\begin{theorem}\label{RES-AD-CONT-W}
Let $(\X,d,\mu)$ be a proper metric measure space satisfying the $\delta$-annular decay property \eqref{delta-AD} for some $\delta\in(0,1]$ and $D_\delta\geq 1$. Suppose that $\varrho$ is a continuous admissible radius function in a bounded domain $\Omega\subset\X$. Then, for any $u\in L^\infty(\Omega)$, any compact set $K\subset\Omega$, any $x,y\in K$ and each $n\in\N$ we have
\begin{equation*}
				\abs{\M^nu(x)-\M^nu(y)}\leq\norm{u}_\infty\,\W_{\mu,K}(d(x,y)),
\end{equation*}
where $\W_{\mu,K}:[0,\diam\Omega]\rightarrow\R$ is given by
\begin{equation}\label{W-AD-CONT}
				\W_{\mu,K}(t)=C\varrho_K^{-\delta}\,\pare{\widehat{\omega}_\varrho(t)}^\delta,
\end{equation}
and $C=C(D_\delta, \mu) > 0$. In particular, the sequence $\braces{\M^n u}_n$ is locally uniformly equicontinuous in $\Omega$.
\end{theorem}

\begin{cor}
Let $(\X,d,\mu)$ be a proper metric measure space satisfying the $\delta$-annular decay property \eqref{delta-AD} for some $\delta\in(0,1]$ and $D_\delta\geq 1$. Suppose that $\varrho$ is a $\gamma$-H\"older continuous admissible radius function in a bounded domain $\Omega\subset\X$, for some $\gamma\in(0,1)$.  Then,

\begin{enumerate}[label=\roman*)]

\item for any $u\in L^\infty(\Omega)$, any compact set $K\subset\Omega$, any $x,y\in K$ and each $n\in\N$ we have
\begin{equation*}
				\abs{\M^nu(x)-\M^nu(y)}\leq\norm{u}_\infty\,\W_{\mu,K}(d(x,y)),
\end{equation*}
where $\W_{\mu,K}:[0,\diam\Omega]\rightarrow\R$ is given by
\begin{equation}\label{W-AD-HOLDER}
				\W_{\mu,K}(t)=C\,\varrho_K^{-\delta}\,t^{\,\gamma\,\delta},
\end{equation}
and $C=C(D_\delta,D_\mu,L)$ where  $L>0$ is the H\"older coefficient of $\varrho$. In particular, the sequence $\braces{\M^n u}_n$ is locally uniformly equicontinuous in $\Omega$.

\item the operator $\M$ sends $L^{\infty}(\Omega )$ to the space $\Lambda_{\,\gamma\delta,\mathrm{loc}}(\Omega)$ of locally $\gamma \delta$-H\"{o}lder continuous functions in $\Omega$, that is  
\begin{equation*}
				\M: L^\infty(\Omega)\rightarrow \Lambda_{\,\gamma\delta,\mathrm{loc}}(\Omega).
\end{equation*}


\end{enumerate}

\end{cor}

\begin{cor}\label{M regularity}
Let $(\X,d,\mu)$ be a prper metric measure space satisfying the $\delta$-annular decay property for some $\delta\in(0,1]$. Suppose that there is $\gamma \in (0,1]$ such that $\varrho$ is a $\gamma$-H\"{o}lder continuous admissible radius function in a bounded domain $\Omega\subset\X$. Then any $u\in L^{\infty} (\Omega )$ satisfying the $0$-mean value property in $\Omega$ with respect to the radius admissible function $\varrho$ (that is, $\M  u = u$) is locally $ \gamma \delta $-H\"{o}lder continuous in $\Omega$. In particular, if  $\delta = \gamma = 1$ then $u$ is locally Lipschitz continuous in $\Omega$. 

\end{cor}



\subsection{Estimates for $\S$}

The following lemma was proven in \cite{LEG-ARC} under the assumption that the admissible radius function is $1$-Lipschitz. Note that, since the operator $\S$ does not depend on any measure, we state it in the context of a metric space $(\X,d$).


\begin{lemma}\label{LEMMA-LGA}
Let $(\X,d)$ be a geodesic metric space and  let $\varrho$ be a continuous admissible radius function in a bounded domain $\Omega\subset\X$. Then, for any $u\in C(\overline{\Omega})$, any compact subset $K\subset\Omega$ and each $x$, $y\in K$ we have
\begin{equation*}
				\abs{\S u(x)-\S u(y)} \leq \omega_{u,\widetilde{K}}\pare{\widehat\omega_\varrho(d(x,y))}.
\end{equation*}
where $\widetilde{K}$, $\omega_{u,\widetilde{K}}$ and $\widehat\omega_\varrho$ are as in  \eqref{hull} , \eqref{def-mod-cont} and \eqref{def-hat-rho}. We have, in particular
\begin{equation}\label{S-est}
				\omega_{\S u,K}(t) \leq \omega_{u,\widetilde{K}}\pare{\widehat\omega_\varrho(t)}.
\end{equation}
\end{lemma}

\begin{proof}
Recalling the definition of $\S u$, \eqref{def-S}, and the elementary formulas
\begin{equation*}
\begin{split}
				\sup_{i\in I}x_i-\sup_{j\in J}y_j &=\sup_{i\in I}\inf_{j\in J}(x_i-y_j), \\
				\inf_{i\in I}x_i-\inf_{j\in J}y_j &=\sup_{j\in J}\inf_{i\in I}(x_i-y_j),
\end{split}
\end{equation*}
we can write
\begin{equation} \label{Su}
				\S u(x)-\S u(y) = \frac{1}{2}\sup_{s\in B_x}\inf_{t\in B_y}(u(s)-u(t)) +\frac{1}{2}\sup_{t\in B_y}\inf_{s\in B_x}(u(s)-u(t)).
\end{equation}
 Note that it may happen that $B_x\not\subset K$ or $B_y\not\subset K$. However, by \eqref{hull}, the inclusion $B_x \cup B_y\subset\widetilde{K}$ holds. 
Then, 
\begin{equation*} 
				\sup_{s\in B_x}\inf_{t\in B_y}(u(s)-u(t))\leq\sup_{s\in B_x}\inf_{t\in B_y}\omega_{u,\widetilde{K}}(d(s,t))\leq\omega_{u,\widetilde{K}}\pare{\sup_{s\in B_x}\inf_{t\in B_y}d(s,t)}.
\end{equation*}
Replacing this term (the other term is analogous) in \eqref{Su} and using that $\omega_{u,\widetilde{K}}$ is concave, we get
\begin{equation*}
				\S u(x)-\S u(y) \leq \omega_{u,\widetilde{K}}\pare{\frac{1}{2}\sup_{s\in B_x}\inf_{t\in B_y}d(s,t)+\frac{1}{2}\sup_{t\in B_y}\inf_{s\in B_x}d(s,t)}.
\end{equation*}
Thus, we need to show that, for any $x,y\in\Omega$, 
\begin{equation}\label{claimLGA}
				\frac{1}{2}\sup_{s\in B_x}\inf_{t\in B_y} d(s,t)+\frac{1}{2}\sup_{t\in B_y}\inf_{s\in B_x} d(s,t)\leq\widehat\omega_\varrho(d(x,y)).
\end{equation}
From \cite[p.282]{LEG-ARC} we get:
\begin{equation} \label{LG est}
\begin{split}
\sup_{t\in B_y}\inf_{s\in B_x} d(s,t) & \leq \max \{d(x,y) + \varrho (x) - \varrho (y) , 0    \}  \\
\sup_{s\in B_x}\inf_{t\in B_y} d(s,t) & \leq \max \{d(x,y) + \varrho (y) - \varrho (x) , 0    \} 
\end{split}
\end{equation}
%
%
%
%
Finally, \eqref{claimLGA} follows from \eqref{LG est} and  \eqref{hat-rho-bounds}. Therefore, this together with \eqref{Su} finishes the proof.

\end{proof}


\section{Iteration of $\T_{\alpha}$}\label{SEC:ITERATION}

As a direct consequence of \Cref{RES-M-CONT} and \Cref{LEMMA-LGA} we have the following result.

\begin{prop}\label{RES-T-CONT}
Let $(\X,d,\mu)$ be a proper, geodesic metric measure space. Suppose that $\Omega\subset\X$ is a bounded domain and let $\varrho$ be a continuous admissible radius function in $\Omega$. If $\alpha \in \R$ then $\T_\alpha:C(\overline{\Omega})\rightarrow C(\overline{\Omega})$. 
\end{prop}

As in the case $\alpha=0$ in which $\T_\alpha$ reduces to $\M$, to go beyond this result we need to take into consideration stronger hypothesis on the measure $\mu$.

\begin{lemma}
Let $(\X,d,\mu)$ be a proper, geodesic metric measure space and let $\Omega \subset \X$ be a bounded domain. Suppose that  $\varrho$ is an admissible radius function in  $\Omega$ and assume that, for every compact set $K\subset\Omega$, a modulus of continuity $\W_{\mu,K}$ is given satisfying \eqref{def-mod-WK-1}. Then, if $|\alpha | \leq 1$,  and  $u\in C(\overline{\Omega})$, the estimate
\begin{equation}\label{T-est}
				\omega_{\T_\alpha u,K}(t) \leq\abs{\alpha}\omega_{u,\widetilde{K}}\pare{\widehat\omega_\varrho(t)} +(1-\alpha)\norm{u}_\infty\W_{\mu,K}(t),
\end{equation}
holds for all $t\in [0, \diam\Omega]$.
\end{lemma}

\begin{proof}
Let $x,y\in K$. Then, since $\T_\alpha=\alpha\S+(1-\alpha)\M$, we get
\begin{equation*}
				\abs{\T_\alpha u(x)-\T_\alpha u(y)} \leq \abs\alpha\abs{\S u(x)-\S u(y)}+(1-\alpha)\abs{\M u(x)-\M u(y)},
\end{equation*}
and \eqref{T-est} is obtained by taking into consideration the estimates \eqref{S-est} and \eqref{def-mod-WK-1}.
\end{proof}

\

The key point for this subsection is the iteration of formula \eqref{T-est}. Note that, in order to obtain estimates for $\T_{\alpha}u$ on the compact set $K$, we need to control  $u$ on $\widetilde K\supset K$, where $\widetilde K$ is given by \eqref{hull}. Thus, when iterating \eqref{T-est}, we need to guarantee some control on the sequence of sets given by successive application of the $\widetilde{(\cdot)}$ operation over the compact set $K$. For that reason, we need to assume that the domain $\Omega\subset\X$ is bounded and we impose the following restriction on $\varrho\,$:
\begin{equation}\label{radii-bounds}
				\lambda\dist(x,\partial\Omega)^\beta\leq\varrho(x)\leq\varepsilon\dist(x,\partial\Omega),
\end{equation}
for each $x\in\Omega$, where $0<\lambda\leq \ell(\Omega)^{1-\beta}\varepsilon$,  $0<\varepsilon<1$ and $\beta\geq 1$. 
We also introduce the following exhaustion of $\Omega$:
\begin{equation}\label{def-Omega}
				K_m :\,= \set{x\in\Omega}{\dist(x,\partial\Omega) \geq (1-\varepsilon)^m},
\end{equation}
for $m\in\N$, where $\varepsilon$ is the constant appearing in  \eqref{radii-bounds}. Hence, $K_1\subset K_2\subset\cdots\Subset\Omega$ and $\displaystyle\lim_{m\rightarrow\infty}K_m=\Omega$ in the sense that, for every $x\in\Omega$, there exists large enough $m_0=m_0(x)\in\N$ such that $x\in K_m$ for all $m\geq m_0$. Moreover, by \eqref{hull} and \eqref{radii-bounds}, it is easy to check that
\begin{equation}\label{inclusion-Omega-tilde}
				\widetilde{K_m}\subset K_{m+1},
\end{equation}
for $m\in\N$. 
%
%
From \eqref{radii-bounds}, we can also control from below the values of $\varrho$ on $K_m$:
\begin{equation} \label{varrho-Omega-m}
				\varrho_{K_m}\geq\lambda\pare{\inf_{K_m}\dist(x,\partial\Omega)}^\beta\geq\lambda(1-\varepsilon)^{m\beta},
\end{equation}
where $\varrho_{K_m}$ is as in \eqref{rho G}.
Replacing $K$ by $K_m$ in \eqref{T-est} and iterating it we can control the oscillation of  $\T_\alpha^n$, for $n\in\N$, as the next lemma shows.  

\begin{lemma}
Let $(\X,d,\mu)$ be a proper, geodesic metric measure space, $\Omega\subset\X$ a bounded domain and let $\varrho$ be a continuous admissible radius function in $\Omega$. Suppose that, for every compact sect $K\subset\Omega$, a modulus of continuity $\W_{\mu,K}$ is given satisfying \eqref{def-mod-WK-1}. Then, for $|\alpha | \leq 1$ and $u\in C(\overline{\Omega})$, the estimate
\begin{multline}\label{Tk-modcont}
				\omega_{\T^n_\alpha u,K_m}(t)  \leq \\  \abs{\alpha}^n\omega_{u,K_{m+n}}\pare{\widehat\omega_\varrho^{(n)}(t)}+ (1-\alpha)\norm{u}_\infty\sum_{j=0}^{n-1}\abs{\alpha}^j\W_{\mu,K_{m+j}}\pare{\widehat\omega_\varrho^{(j)}(t)}
\end{multline}
holds for each $n$, $m\in \N$ and  every $t\in [0, \diam \Omega]$. 
\end{lemma}

\begin{proof}
Since $\widetilde{K_m} \subset K_{m+1}$, we get from \eqref{T-est} 
\begin{equation*}
				\omega_{\T_\alpha u,K_m}(t) \leq\abs{\alpha}\omega_{u,K_{m+1}}\pare{\widehat\omega_\varrho(t)} +(1-\alpha)\norm{u}_\infty\W_{\mu,K_{m+1}}(t)
\end{equation*}
for each $t\in [0, \diam\Omega]$. Now, iteration of this inequality gives \eqref{Tk-modcont}.
\end{proof}

To get equicontinuity of the sequence $\displaystyle \braces{ \T_{\alpha}^n u }_n $,  we need to add some extra condition.

\begin{lemma}\label{RES-LEMMA-T-LOC-UNI-EQC}
Let $(\X,d,\mu)$ be a proper, geodesic metric measure space with a continuous admissible radius function $\varrho$ in a bounded domain $\Omega\subset\X$. Suppose that, for every compact sect $K\subset\Omega$, a modulus of continuity $\W_{\mu,K}$ is given satisfying \eqref{def-mod-WK-1}. 
Assume also that 
\begin{equation}\label{alpha-cond}
				\abs\alpha \limsup_{j\rightarrow\infty}\pare{\W_{\mu,K_j}\pare{\diam\Omega}}^{1/j} < 1 
\end{equation}
Then for any $u\in C(\overline{\Omega})$, the sequence $\braces{\T^n_\alpha u}_n$ is locally uniformly equicontinuous in $\Omega$.
\end{lemma}

\begin{proof}
Fix $m \in \N$. Regarding the first term in the right-hand side of \eqref{Tk-modcont} we note that, since $\widehat\omega_\varrho(t)\leq \diam\Omega$ for each $t\in [0, \diam\Omega ]$, then
\begin{equation*}
				\abs{\alpha}^n\omega_{u,K_{m+n}}\pare{\widehat\omega^{(n)}(t)}\leq \abs{\alpha}^n\omega_{u,\Omega}\pare{\diam\Omega}\xrightarrow[\ n\rightarrow\infty\ ]{} 0.
\end{equation*}
Thus, 
\begin{equation}\label{conv unif}
\braces{t\mapsto\abs{\alpha}^n\omega_{u,\Omega}\pare{\widehat\omega^{(n)}_\varrho(t)}}_n \xrightarrow[\ n\rightarrow\infty\ ]{} 0
\end{equation}
uniformly in $[0, \diam\Omega]$ as $n\rightarrow\infty$. Consequently there exists a common modulus of continuity $\mathcal{F}_1$ for the sequence \eqref{conv unif}. Now we focus on the series in \eqref{Tk-modcont}. Note that
\begin{equation*}
\W_{\mu,K_{m+j}}\pare{\widehat \omega_\varrho^{(j)}(t)} \leq \W_{\mu,K_{m+j}}\pare{\diam\Omega}
\end{equation*}
for all $t\in [0, \diam\Omega]$. Then, since
\begin{equation*}
\limsup_{j\rightarrow\infty}\pare{\W_{\mu,K_{m+j}}\pare{\diam\Omega}}^{1/j}= \limsup_{j\rightarrow\infty}\pare{\W_{\mu,K_{m+j}}\pare{\diam\Omega}}^{1/(m+j)},
\end{equation*}
it follows from \eqref{alpha-cond} that 
\begin{equation*}
|\alpha | \limsup_{j\rightarrow\infty}\pare{\W_{\mu,K_{m+j}}\pare{\diam\Omega}}^{1/j} < 1,
\end{equation*}
so the root test implies that the series 
\begin{equation*}
				\sum_{j=0}^\infty\abs{\alpha}^j\W_{\mu,K_{m+j}}\pare{\widehat\omega_\varrho^{(j)}(t)}<\infty
\end{equation*}
converges uniformly in $[0,\diam\Omega]$. In particular, there exists another modulus of continuity for the series, say $\mathcal{F}_2$. Summarizing:
\begin{equation*}
				\omega_{\T^n_\alpha u,K_{m}}(t) \\ \leq \mathcal{F}_1(t)+ (1-\alpha)\norm{u}_\infty\mathcal{F}_2(t).
\end{equation*}
Since $m$ is arbitrary and the right-hand side of the previous inequality does not depend on $n\in\N$, the proof is finished.
\end{proof}

\begin{theorem}\label{RES-EQUICONT}
Let $(\X,d,\mu)$ be a proper, geodesic metric measure space satisfying the $\delta$-annular decay property \eqref{delta-AD} for some $\delta\in(0,1]$.  Let $\abs\alpha<1$ and suppose that $\varrho$ is a continuous admissible radius function in a bounded domain $\Omega\subset\X$ satisfying \eqref{radii-bounds} with $0<\lambda\leq \ell(\Omega)^{1-\beta}\varepsilon$. Assume also that 
\begin{flalign}
				0  < \, & \, \varepsilon \, < 1-\abs{\alpha}, \label{epsilon-alpha} \vspace{0.2cm} \\ 1 \leq \, & \, \beta \, < \frac{\log\displaystyle\frac{1}{\abs{\alpha}}}{\, \log\displaystyle\frac{1}{1-\varepsilon}\, }. \label{alpha-beta-epsilon 3}
\end{flalign}
Then, for any $u\in C(\overline{\Omega})$, the sequence of iterates $\braces{\T^n_\alpha u}_n$ is locally uniformly equicontinuous in $\Omega$.
\end{theorem}

\begin{proof}
We only need to check that the assumptions in \Cref{RES-LEMMA-T-LOC-UNI-EQC} are satisfied. By \Cref{RES-AD-CONT-W}, for any compact set $K\subset\Omega$, we can choose $\W_{\mu,K}$ as in \eqref{W-AD-CONT} for any compact set $K \subset \Omega$. 
Thus, after replacing $K$ by $K_j$ and $t$ by $\diam\Omega$ and recalling  that $\widehat\omega_\varrho(\diam\Omega)=\diam\Omega$, we get,
\begin{equation*}
				\pare{\W_{\mu,K_j}(\diam\Omega)}^{1/j}=\pare{C(\diam\Omega)^\delta}^{1/j}\varrho_{K_j}^{\,-\delta/j},
\end{equation*}
and by \eqref{varrho-Omega-m},
\begin{equation*}
				\pare{\W_{\mu,K_j}(\diam\Omega)}^{1/j}\leq\pare{\frac{C (\diam \Omega )^{\delta}}{\lambda^{\delta}}}^{1/j}(1-\varepsilon)^{\,-\delta\,\beta}.
\end{equation*}
Taking limits we get
\begin{equation*}
				\limsup_{j\rightarrow\infty}\pare{\W_{\mu,K_j}(\diam\Omega)}^{1/j}\leq (1-\varepsilon)^{\,-\delta\,\beta}.
\end{equation*}
On the other hand, by \eqref{epsilon-alpha}  we have $\abs \alpha < 1 - \varepsilon \leq (1-\varepsilon)^{\,\delta\,\beta}$ so condition \eqref{alpha-cond} follows and the sequence $\braces{\T^n_\alpha u}_n$ is locally uniformly equicontinuous in $\Omega$ by \Cref{RES-LEMMA-T-LOC-UNI-EQC}. 
\end{proof}


\section{Regularity of solutions}


In this section we give regularity results for functions $u\in C(\overline{\Omega})$  satisfying the $\alpha$-mean value property (that is, solutions of the functional equation $\T_{\alpha}u = u$) with respect to an admissible radius function in a bounded domain $\Omega \subset \X$. When $\alpha=0$, then $\T_0=\M$ and the regularity of such solutions was already  obtained in \Cref{M regularity}. 
However, the case $\alpha \neq 0$ is more delicate and stronger assumptions on the radius function $\varrho$ are needed, as we have already seen in \Cref{SEC:ITERATION}.
\\

We focus our attention on inequality \eqref{Tk-modcont}. Since the continuous function $u$ is assumed to be a fixed point of the operator $\T_\alpha$, after replacing $\T_\alpha^n u$ by $u$, we are allowed to pass to the limit when $n\rightarrow\infty$. From \eqref{Tk-modcont} we get
\begin{equation}\label{u-fixed-point-series}
				\omega_{u,K_m}(t)\leq(1-\alpha)\norm{u}_\infty\sum_{j=0}^{\infty}\abs{\alpha}^j\W_{\mu,K_{m+j}}\pare{\widehat\omega_\varrho^{(j)}(t)},
\end{equation}
for $t\in [0, \diam\Omega]$, where $m\in \N$ is fixed. Therefore, the series in \eqref{u-fixed-point-series} will provide the information about the regularity of the solution $u$.
The following is our main regularity result. 

\begin{theorem} \label{main}
Let $(\X,d,\mu)$ be a proper, geodesic metric measure space satisfying the
$\delta$-annular decay condition for some $\delta\in(0,1]$ and let $\Omega \subset \X$ be a bounded domain. Suppose that  $\varrho$ is a Lipschitz admissible radius function in $\Omega$ with Lipschitz constant $L\geq 1$ such that
\begin{equation*}
				\lambda\dist(x,\partial\Omega)^\beta\leq\varrho(x)\leq\varepsilon\dist(x,\partial\Omega),
\end{equation*}
for all $x\in \Omega$, where $0 < \lambda \leq \ell(\Omega )^{1-\beta}\varepsilon$ and $\ell (\Omega)$ is given by \eqref{eq-L-Omega}. Assume also that  
\begin{equation*}
\begin{split}
				\abs{\alpha} & < L^{-1},  \\
				0 <  \ \varepsilon ~& < 1-L\abs{\alpha},
\end{split}
\end{equation*}
and choose $\beta$ so that 
\begin{equation} \label{alpha-L}
1  \, \leq \, \beta<\frac{\log\displaystyle\frac{1}{L\abs{\alpha}}}{\, \log\displaystyle\frac{1}{1-\varepsilon}\, }. 
\end{equation}
%
Then any $u\in C(\overline{\Omega})$ verifying the $\alpha$-mean value property in $\Omega$ with respect to $\varrho$ (that is, $\T_{\alpha} u = u$) is locally $\delta$-H\"{o}lder continuous in $\Omega$. In particular, if $\delta = 1$ then $u$ is locally Lipschitz continuous in $\Omega$. 
\end{theorem}


\begin{proof}
By assumption, $\varrho$ is $L$-Lipschitz, therefore we have $\widehat\omega_\varrho(t)=\min\braces{L t,\diam\Omega}$ 
Iterating we get the inequality $ \displaystyle \widehat\omega_\varrho^{(j)}(t)\leq L^jt$ for each $t\in [0, \diam\Omega]$ and each $j\in\N$. Moreover, since $\mu$ satisfies the $\delta$-annular decay property \eqref{delta-AD}, from \eqref{W-AD-HOLDER} together with \eqref{varrho-Omega-m} we get
\begin{equation*}
				\W_{\mu,K_{m+j}}(t)\leq \frac{C\,t^\delta}{\lambda^\delta(1-\varepsilon)^{(m+j)\beta\,\delta}}
\end{equation*}
for some constant $C=C(D_\delta,D_\mu,L)\geq 1$. Replacing all this in \eqref{u-fixed-point-series} we obtain the following estimate:
\begin{equation*}
				\omega_{u,K_m}(t)\leq \frac{C(1-\alpha)\norm{u}_\infty}{\lambda^\delta(1-\varepsilon)^{m\,\beta\,\delta}}\pare{\sum_{j=0}^{\infty}\pare{\frac{L^\delta\abs{\alpha}}{(1-\varepsilon)^{\beta\,\delta}}}^j}t^\delta.
\end{equation*}
Now observe that \eqref{alpha-L} implies the convergence of the above series and, consequently, the desired H\"{o}lder regularity estimate. 
\end{proof}

In the particular case that $\beta = 1$ we obtain the following corollary. 
\begin{cor} \label{reg beta=1}
Let $(\X,d,\mu)$ be a proper, geodesic metric measure space satisfying the
$\delta$-annular decay condition for some $\delta\in(0,1]$ and let $\Omega \subset \X$ be a bounded domain. Suppose that  $\varrho$ is a Lipschitz admissible radius function in $\Omega$ with Lipschitz constant $L \geq 1$ such that
\begin{equation*}
				\lambda\dist(x,\partial\Omega)\leq\varrho(x)\leq\varepsilon\dist(x,\partial\Omega),
\end{equation*}
for all $x\in \Omega$, where $0 < \lambda \leq \varepsilon$. Assume also that  
\begin{equation*}
\begin{split}
				\abs{\alpha} & < L^{-1},  \\
				0 <  \ \varepsilon ~& < 1-L\abs{\alpha}.
\end{split}
\end{equation*}
Then any $u\in C(\overline{\Omega})$ verifying the $\alpha$-mean value property in $\Omega$ with respect to $\varrho$ (that is, $\T_{\alpha} u = u$) is locally $\delta$-H\"{o}lder continuous in $\Omega$. In particular, if $\delta = 1$ then $u$ is locally Lipschitz continuous in $\Omega$. 
\end{cor}



\end{document}